\documentclass[twoside,12pt]{article}

\usepackage[utf8]{inputenc}
\usepackage{amssymb,amsmath,amsthm,subfiles,mathtools,mathrsfs,comment,dsfont,mathrsfs,graphicx,enumerate,caption,wrapfig,float,hyperref,titlesec}
\usepackage[dvipsnames]{xcolor}

\makeatletter
\newcommand*\bigcdot{\mathpalette\bigcdot@{1}}
\newcommand*\bigcdot@[2]{\mathbin{\vcenter{\hbox{\scalebox{#2}{$\m@th#1\bullet$}}}}}
\makeatother

\makeatletter
\newcommand*{\rom}[1]{\expandafter\@slowromancap\romannumeral #1@}
\makeatother

\usepackage{titleps,lipsum}

\usepackage{caption}
\usepackage{subcaption}

\usepackage{afterpage}

\usepackage[T1]{fontenc}

\author{
 \scshape Lian Haeming \\
  \textit{Queen Mary University of London}
}

\graphicspath{{\string~/Desktop/Tex Files/Lyapunov rectangle boxes.jpg}}
    
\hypersetup{
    colorlinks=true,
    linkcolor=MidnightBlue,
    citecolor = MidnightBlue,
    linktoc = all
}

\usepackage{xcolor}
\pagecolor{white}
\usepackage[margin=0.72in]{geometry}
\numberwithin{equation}{section}
\title{\scshape \bfseries \Large on the real eigenvalues of the non-hermitian anderson model}
\date{}
\newcommand\restr[2]{{
  \left.\kern-\nulldelimiterspace 
  #1 
  \vphantom{\big|} 
  \right|_{#2} 
  }}

\newcommand\E{\mathds{E}}
\newcommand\R{\mathbb{R}}
\newcommand\C{\mathbb{C}}

\makeatletter
\newcommand{\newreptheorem}[2]{\newtheorem*{rep@#1}{\rep@title}\newenvironment{rep#1}[1]{\def\rep@title{#2 \ref*{##1}}\begin{rep@#1}}{\end{rep@#1}}}
\makeatother

\def\restrict#1{\raise-.5ex\hbox{\ensuremath|}_{#1}}
\newtheorem{thm}{\normalfont\bfseries Theorem}
\newtheorem{lemma}[thm]{\normalfont\bfseries  Lemma}

\newtheorem{rmk}[thm]{\normalfont\bfseries Remark}
\newtheorem{prop}[thm]{\normalfont\bfseries Proposition}

\newreptheorem{theorem}{\normalfont\bfseries Theorem}
\newreptheorem{lemma}{\normalfont\bfseries Lemma}
\newreptheorem{prop}{\normalfont\bfseries Proposition}
\newreptheorem{cor}{\normalfont\bfseries Corollary}

\makeatletter
\let\hdrtitle\@title
\newpagestyle{main}{%
  \sethead[][\scshape{Quasiballistic motion}][\thepage] 
  {}{\scshape{L. Haeming}}{\thepage}} 
\makeatother
\titlespacing*{\section}{0pt}{4ex}{1.5ex}

\titleformat{\section}[block]{\color{black}\scshape\filcenter}{\thesection.}{0.5em}{}

\begin{document}
\renewcommand{\abstractname}{\vspace{-\baselineskip}}
\maketitle

\begin{abstract}
We study the non-Hermitian Anderson model on the ring. We provide the exact rate of decay of the sensitivity of the eigenvalues to the non-Hermiticity parameter $g$, on the logarithmic scale, as the Lyapunov exponent  minus the non-Hermiticity parameter. Namely, for $0 < g < \gamma(\lambda_{0})$ we show that $-\frac{1}{n}\log|\lambda_{g}-\lambda_{0}|\sim \gamma(\lambda_{0})-g$ and that the eigenvalue remains real for all such $g$. This provides an alternative proof to that of Goldsheid and Sodin that the perturbed eigenvalue remains real and specifies the exact rate at which the eigenvalue is exponentially close to the unperturbed eigenvalue. 
\end{abstract}

\section*{Introduction}

Let $v_{1},v_{2},\dots$ be a bounded sequence of independent identically  distributed random variables. We study the one-dimensional non-Hermitian Anderson model with periodic boundary conditions 
\begin{equation}\label{non:model}
(H_{n}(g)\psi)_{k} = e^{-g}\psi_{k-1}+e^{g}\psi_{k+1}+v_{k}\psi_{k}, \quad 1\leq k\leq n, \quad \psi_{0}=\psi_{n}, \,\, \psi_{1}=\psi_{n+1}.
\end{equation}

This model was first introduced by Hatano and Nelson \cite{HaNe96,HaNe98} to describe the reaction of an Anderson-localised quantum particle on a ring to a constant imaginary vector field. Our focus is on the behaviour of the individual (real) eigenvalues $\lambda_{j}(g)$ of \eqref{non:model} in the limit $n\rightarrow\infty$. Since the operator $H_{n}(g)$ is non-Hermitian for $g>0$, the eigenvalues are not necessarily real. The numerical work of Hatano and Nelson suggests that all of the eigenvalues remain in the real axis for sufficiently small $0\leq g<g_{1}$;  all eigenvalues move out of the real axis and align along a smooth curve on the complex plane for sufficiently large $g>g_{2}$; and we see a combination of the two for $g_{1}<g<g_{2}$. 

For a more general setting which includes \eqref{non:model} as a special case, Goldsheid and Khoruzhenko \cite{GoKh98, GoKh00,GoKh03} showed that the behaviour of the eigenvalues of \eqref{non:model} depends on the Lyapunov exponent associated with the Hermitian operator, 
\begin{equation}\label{non:lyapunov}
\gamma(E) = \lim_{n\rightarrow\infty} \frac{1}{n}\E \log \|A_{E,n}\cdots A_{E,1}\|, \quad A_{E,k}=
\begin{pmatrix}
E-v_{k}& -1 \\
1 & 0
\end{pmatrix}.
\end{equation} It was shown that on an event of asymptotically full probability: The perturbed eigenvalues remain in the vicinity of the real axis when the non-Hermiticity parameter is less than the Lyapunov exponent $g<\gamma(\lambda_{j}(0))$. And when it is greater than the Lyapunov exponent $g>\gamma(\lambda_{j}(0))$ they spread out in a regular fashion along certain polynomial curves in the complex plane, which tend to the level curve of the Lyapunov exponent $\{z\in\C:\gamma(z)=g\}$ as $n\rightarrow\infty$. For existence of the limit \eqref{non:lyapunov}, see e.g.\ the lecture notes of Viana \cite{Vi14}. Furstenberg’s theorem \cite{Fu63} shows that the Lyapunov exponent is positive on the real line $\min_{\R}\gamma>0$.

\begin{thm}\label{non:thm}
Let $v_{1},\dots,v_{n}$ be a sequence of i.i.d.\ random variables.  For any $0<\varepsilon<\varepsilon_{0}$ there exist $c=c(\varepsilon)$ and $N_{0}=N_{0}(\varepsilon)$  such that if $n>N_{0}$, then with probability $>1-e^{-cn}$: For each $1\leq j\leq n$, if $0\leq g\leq \gamma(\lambda_{j}(0))-\varepsilon$ then the $j$-th eigenvalue is real $\lambda_{j}(g)\in\R$  and  satisfies 
\begin{equation}\label{non:10}
(1-e^{-ng})^{2}e^{-(\gamma(\lambda_{j}(0))-g+\varepsilon)n}\leq |\lambda_{j}(g)-\lambda_{j}(0)|<e^{-(\gamma(\lambda_{j}(0))-g-\varepsilon)n}.
\end{equation}
\end{thm}

The question of whether the eigenvalues are truly real for $g<\gamma(\lambda_{j}(0))-\varepsilon$ was established by Goldsheid and Sodin \cite[Theorem 1]{GoSo18}, who showed that on an event of asymptotically full probability: For each $1\leq j\leq n$, if the non-Hermiticity parameter is less than the Lyapunov exponent $0\leq g< \gamma(\lambda_{j}(0))-\varepsilon$, 
then the perturbed eigenvalue is real $\lambda_{j}(g)\in\R$ and is exponentially close to the unperturbed eigenvalue $|\lambda_{j}(g)-\lambda_{j}(0)|<e^{-c_{0}(\varepsilon)n}$, for an unspecified positive constant $c_{0}(\varepsilon)>0$. Our upper bound shows $c_{0}(\varepsilon)=c_{0}(\varepsilon,g,j)\approx \gamma(\lambda_{j}(0))-g$ and our lower bound shows that it is optimal. The factor of $(1-e^{-ng})^{2}$ in the lower bound is chosen so that the LHS matches the RHS and comes from $2\cosh(ng)-2 = e^{ng}(1-e^{-ng})^{2}$. Equality in the LHS of \eqref{non:10} occurs for $g=0$. Our theorem as well as \cite{GoSo18} requires the potential to be i.i.d.

\section*{Acknowledgements}
It is a pleasure to thank Mira Shamis for proposing this work and Sasha Sodin for his support towards its completion.  This work was supported by an EPSRC PhD fellowship and supported in part by an EPSRC research grant (EP/X018814/1) and by a Philip Leverhulme Prize of the Leverhulme Trust (PLP-2020-064).

\section{Proof of the theorem}

There are two main steps to the proof of the theorem. The first has been established earlier by \cite[Theorem 1]{GoSo18} which is to show that the eigenvalue is real $\lambda_{j}(g)\in\R$ for $0\leq g\leq \gamma(\lambda_{j}(0))-\varepsilon$. The second is to obtain the bounds \eqref{non:10}, which requires the first step as an input. We state the first step as Lemma \ref{non:part1} and provide an alternative proof of it in the final section. The remainder of the present section is dedicated only to the proof of the estimates \eqref{non:10}, with Lemma \ref{non:part1} taken as an input. The probability of the event of Lemma \ref{non:part1} was not mentioned in \cite{GoSo18}. 

\begin{lemma}[{{\cite[Theorem 1]{GoSo18}}}]\label{non:part1}
Let $v_{1},\dots,v_{n}$ be a sequence of i.i.d.\ random variables. For any $0<\varepsilon<\varepsilon_{0}$ there exist $N_{0}=N_{0}(\varepsilon)$ and $c=c(\varepsilon)$ such that if $n>N_{0}$, then with probability $>1-e^{-cn}$: For every $1\leq j\leq n$, if $0\leq g\leq \gamma(\lambda_{j}(0))-\varepsilon$ then the eigenvalue is real $\lambda_{j}(g)\in\R$. 
\end{lemma}

The trace $\Delta_{n}(E) = \text{Tr}(A_{E,n}\cdots A_{E,1})$ of the $n$-step transfer matrices \eqref{non:lyapunov} is the key instrument in our argument for  \eqref{non:10}. The trace $\Delta_{n}$ is a polynomial in $E\in\R$ of degree $n\geq 1$ and has real coefficients. The trace satisfies  $|\Delta_{n}(E_{j}')|\geq 2$ at its turning points $E_{j}'\in\R$. The pre-image of the interval $[-2,2]$ under the trace is therefore equal to the union 
$$
\Delta_{n}^{-1}([-2,2])=\bigcup_{j=1}^{n}B_{n}^{(j)}
$$ of $n$ closed intervals $B_{n}^{(j)}\subset\R$, known as bands, which have mutually disjoint interiors. 

The proof of the upper bound \eqref{non:10} is a consequence of Lemma \ref{non:Last} and Proposition \ref{non:bandwidth}. Lemma \ref{non:Last} is an extension (see the proof below) of two inequalities of Last \cite[Lemma 1]{La94} and allows to estimate the distance from the eigenvalue to the root in terms of the corresponding bandwidth multiplied by the value of the trace at the eigenvalue (see \eqref{non:15}).  
 
 \begin{lemma}\label{non:Last}
Let $E_{1}\leq E_{2}\leq \dots$ and $E'_{1}\leq E'_{2}\leq \dots$ denote the roots and turning points of the trace $\Delta_{n}$, respectively. We have 
$$
|E-E_{j}|< \frac{e}{1+\sqrt{5}}|\Delta_{n}(E)||B_{n}^{(j)}|
$$
for $E\in(E'_{j-1},E'_{j})$ if $2\leq j\leq n-1$, for $E\in(E_{1},E'_{1})$ if $j=1$ and $E\in (E'_{n-1},E_{n})$ if $j=n$. 
\end{lemma}

Proposition \ref{non:bandwidth} is an upper bound on the bandwidths in terms of the Lyapunov exponent \cite[Theorem 2]{Ha22} and follows from the fact that the eigenvectors of the Floquet matrix $H_{n}(ig)$, $g\geq0$ (with i.i.d.\ potential) decay exponentially away from a centre, with rate of exponential decay given by the Lyapunov exponent. For details see e.g.\ \cite[Lemma 2.1]{Ha22}).   

\begin{prop}[{{\cite[Theorem 2]{Ha22}}}]\label{non:bandwidth}
Let $v_{1},\dots,v_{n}$ be a sequence of i.i.d.\ random variables. For any $\varepsilon>0$ there exist $c=c(\varepsilon)$ and $N_{0} = N_{0}(\varepsilon)$ such that if $n>N_{0}$, then with probability $>1-e^{-cn}$:  For each $1\leq j\leq n$,
\begin{equation}\label{non:11}
|B_{n}^{(j)}| <e^{-(\gamma(\lambda_{j}(0)) -\varepsilon)n}.
\end{equation} 
\end{prop}

\begin{rmk}
The upper bound \eqref{non:11} is sharp in the sense that a matching lower bound $|B_{n}^{(j)}|>e^{-(\gamma(\lambda_{j}(0)) +\varepsilon)n}$ also holds  \cite[Theorem 1]{Ha22}. 
\end{rmk}

\begin{proof}[Proof of Theorem \ref{non:thm}]
The first part of the theorem, i.e., that the eigenvalue is real for $0\leq g\leq \gamma(\lambda_{j}(0))-\varepsilon$, follows from Lemma 
\ref{non:part1}, which is proved in the final section. Order the eigenvalues from left to right $\lambda_{1}(0)\leq \lambda_{2}(0)\leq \dots$ and let $E_{j}, E'_{j}$ be as stated in Lemma \ref{non:Last}. 

By computing the characteristic polynomial of the matrix which describes the action of the operator $H_{n}(g)$ (after the usual gauge transformation by the diagonal matrix $\text{Diag}(1,e^{g},\dots,e^{ng})$) and evaluating it at the eigenvalue $\lambda_{j}(g)$, one obtains the exact equality 
\begin{equation}\label{non:discriminant}
|\Delta_{n}(\lambda_{j}(g))|=2\cosh(ng).
\end{equation} 
Since the hyperbolic cosine is monotonic for $g\geq0$, and the trace is monotonic in between two turning points (as well as below the leftmost turning point $E'_{1}$, and the rightmost turning point $E'_{n-1}$); for any $1\leq j\leq n$, the eigenvalue $\lambda_{j}(g)$ is monotonic in $0\leq g\leq \gamma(\lambda_{j}(0))-\varepsilon$. 

\emph{Upper bound.} \eqref{non:discriminant} implies that the eigenvalue $\lambda_{j}(g)$ moves away (to the left or to the right) from the root $E_{j}$, therefore for each $1\leq j\leq n$ and $0\leq g\leq \gamma(\lambda_{j}(0))-\varepsilon$, the distance to the root is larger than the distance to the unperturbed eigenvalue, 
\begin{equation}\label{non:04}
|\lambda_{j}(g)-\lambda_{j}(0)|<|\lambda_{j}(g)-E_{j}|.
\end{equation}

According to Lemma \ref{non:Last} we must consider the inner eigenvalues (with indices $2\leq j\leq n-1$), separately from the outer eigenvalues (with indices $j=1,n$). 

For an inner eigenvalue with $0\leq g\leq \gamma(\lambda_{j}(0))-\varepsilon$ we have $\lambda_{j}(g)\in(E'_{j-1},E'_{j})$; so Lemma \ref{non:Last},  \eqref{non:discriminant} and \eqref{non:04} imply 
\begin{equation}\label{non:15}
|\lambda_{j}(g)-\lambda_{j}(0)|<2\cosh(ng)|B_{n}^{(j)}|.
\end{equation}
 The upper bound \eqref{non:10} for the inner eigenvalues then follows from Proposition \ref{non:bandwidth} applied to \eqref{non:15}.

Let us now consider the leftmost eigenvalue $\lambda_{1}$. Take $0\leq g\leq \gamma(\lambda_{1}(0))-\varepsilon$. If the eigenvalue moves to the right, i.e., $\lambda_{1}(0)>E_{1}$, then $\lambda_{1}(g)\in(E_{1},E'_{1})$. Lemma \ref{non:Last} still applies in this case and hence so does \eqref{non:15} (and \eqref{non:10}). Lemma \ref{non:Last} does not hold in the case that the eigenvalue moves to the left, i.e.,  $\lambda_{1}(0)<E_{1}$. In this case, define the leftmost two points $\lambda^{-}_{1}(g)<\lambda^{+}_{1}(g)$ which satisfy 
$$|\Delta_{n}(\lambda_{1}^{\pm}(g))|=2\cosh(ng).$$ The behaviour of the solution $\lambda_{1}^{+}(g)$ describes the scenario (described above) in which $\lambda_{1}(0)>E_{1}$ whereas $\lambda_{1}^{-}(g)$ describes the scenario in which $\lambda_{1}(0)<E_{1}$. For each $0\leq g\leq \gamma(\lambda_{1}(0))-\varepsilon$, we have $$|\lambda_{1}^{-}(g)-E_{1}|\leq |\lambda_{1}^{+}(g)-E_{1}|$$ since the derivative $\Delta'_{n}$ is strictly monotonic below $E'_{1}$. This argument holds analogously for $j=n$.

\emph{Lower bound.} Let  $0\leq g\leq \gamma(\lambda_{j}(0))-\varepsilon$ so that $\lambda_{j}(g)\in \R$ and suppose that the $j$-th eigenvalue $\lambda_{j}(g)$ has positive derivative. By the mean value theorem,
$$
2\cosh(ng)-2=|\Delta_{n}(\lambda_{j}(g))-\Delta_{n}(\lambda_{j}(0))|\leq |\lambda_{j}(g)-\lambda_{j}(0)|\max_{E\in[\lambda_{j}(0),\lambda_{j}(g)]}|\Delta'_{n}(E)|. 
$$ It suffice to show that for any $\varepsilon>0$, there exist $c=c(\varepsilon)$ and $N_{0}=N_{0}(\varepsilon)$ such that if $n>N_{0}$, then with probability $>1-e^{-cn}$: 
\begin{equation}\label{non:16}
\max_{E\in[\lambda_{j}(0),\lambda_{j}(g)]}|\Delta'_{n}(E)| \leq e^{(\gamma(\lambda_{j}(0))+\varepsilon)n}. 
\end{equation}

Indeed, by the large deviation theorem of Le Page \cite{LP82} (see also \cite[Lemma 2.1]{GoSo18}), for any $\varepsilon>0$, there exist $c=c(\varepsilon)$ and $N_{0}=N_{0}(\varepsilon)$ such that if $n>N_{0}$, then with probability $>1-e^{-cn}$: 
\begin{equation}\label{non:07}
\log\|A_{E,n}\cdots A_{E,1}\|-n\gamma(E)\leq n\varepsilon. 
\end{equation}
In \cite[Lemma 1.3]{Ha22} we show that \eqref{non:07} can be upgraded to hold uniformly in the energy on compact sets. Since the trace of the transfer matrix is at most twice its norm, one can obtain a similar uniform upper bound for the trace itself. One can then obtain a similar upper bound for the derivative of the trace (see e.g.\  \cite[Theorem 1]{Ha22}) by means of applying the Markov inequality (for a polynomial $p_{n}$ of degree $n$):
\begin{equation}\label{non:markov}
\max_{x\in [a,b]}|p_{n}'(x)|\leq \frac{2n^{2}}{b-a}\max_{x\in [a,b]}|p_{n}(x)|
\end{equation} locally, and use the continuity of the Lyapunov exponent due to Le Page \cite{LP83}. In particular, 
\begin{equation}\label{non:13}
\sup_{E\in K} (\log |\Delta'_{n}(E)| - n\gamma(E)) \leq n\varepsilon
\end{equation} where $K$ is a closed interval chosen to contain all spectra $\bigcup_{n\geq1}\sigma(H_{n}(0))\subset K$ so that $N_{0}(\varepsilon)$ doesn't depend on $j$. The upper bound \eqref{non:10} implies $\max_{0\leq g\leq \gamma(\lambda_{j}(0))-\varepsilon}|\lambda_{j}(g)-\lambda_{j}(0)|<e^{-\varepsilon n}$ and then continuity of the Lyapunov exponent implies $\max_{E\in[\lambda_{j}(0),\lambda_{j}(g)]}\gamma(E)\leq \gamma(\lambda_{j}(0))+\varepsilon$. \eqref{non:16} follows and so does the theorem, after appropriately scaling the errors in the exponents. 
\end{proof}

\begin{proof}[Proof of Lemma \ref{non:Last}]
An extension (for details on this extension see e.g. \cite[(5.9)]{LaSh16}, where they use $E_{\nu}$ to denote the $\nu$-th root of the discriminant $\Delta_{n}$, instead of $E_{j}$) of the first estimate of Last \cite[(3.13)]{La94} on the trace states
$$
|E-E_{j}|\leq \frac{e|\Delta_{n}(E)|}{|\Delta'_{n}(E_{j})| }
$$
which holds for $E\in(E'_{j-1},E'_{j})$ if $2\leq j\leq n-1$ and for $E\in(E_{1},E'_{1}) $ or $ E\in (E'_{n-1},E_{n})$ if $j=1$ or $j=n$, respectively. The proof of the second estimate of Last \cite[(3.29)]{La94} can be similarly repeated for the bands $B_{n}^{(j)}$ in the pre-image $\Delta_{n}^{-1}([-2,2])$ to obtain 
\begin{equation}\label{non:03}
|\Delta_{n}'(E_{j})|\geq \frac{1+\sqrt{5}}{|B_{n}^{(j)}|}
\end{equation}
for every $1\leq j\leq n$. 
\end{proof}

\section{Lemma \ref{non:part1}}

The general idea is to show that eigenvalue spacing (Lemma \ref{non:spacings}, below) and the upper bound on the bandwidths (Proposition \ref{non:bandwidth}) both imply that the turning points of the trace are sufficiently large: 
\begin{equation}\label{non:12}
|\Delta_{n}(E'_{j-1})|, |\Delta_{n}(E'_{j})|  > e^{(\gamma(\lambda_{j}(0))-\varepsilon)n} , \quad \, 2\leq j\leq n-1
\end{equation} 
so that the solution $\lambda_{j}(g)\in (E'_{j-1},E'_{j})$ to \eqref{non:discriminant} is  real for every $0\leq g\leq \gamma(\lambda_{j}(0))-\varepsilon$.

\begin{lemma}\label{non:spacings}
Let $v_{1},\dots,v_{n}$ be a sequence of i.i.d.\ random variables. For any $0<\varepsilon<\varepsilon_{0}$ there exist $N_{0} = N_{0}(\varepsilon)$ and $c=c(\varepsilon)$ such that if $n>N_{0}$ then with probability $>1-e^{-cn}$: 
$$\min_{j\neq j'}|\lambda_{j}(0)-\lambda_{j'}(0)|>e^{-\varepsilon n}.$$
\end{lemma}

\begin{rmk}
If the cumulative distribution function corresponding to the distribution of the random variable $v_{1}$ is uniformly Hölder of order $\frac{1}{2}+\delta$ then a conclusion similar to the one in Lemma \ref{non:spacings}, for arbitrary dimensions, follows from the Minami estimate \cite{Mi96}. For less regular distributions, Bourgain \cite{Bo14} showed that in dimension one and with Dirichlet boundary conditions, the eigenvalues separate. In \cite{Ha22} we show that the arguments of Bourgain hold for periodic boundary condition. 
\end{rmk}

\begin{proof}[Proof of Lemma \ref{non:part1}]
 Let $E_{j},E'_{j}$ be as stated in Lemma \ref{non:Last}. From Proposition \ref{non:bandwidth}, Lemma \ref{non:spacings} and positivity of the Lyapunov exponent, it follows that for any $0<\varepsilon<\varepsilon_{0}$, there exist $c(\varepsilon)>0$ and $N_{0}=N_{0}(\varepsilon)$ such that if $n>N_{0}$, then with probability $>1-e^{-cn}$: 
\begin{equation}\label{non:05}
\min_{j\neq j'}|E_{j}-E_{j'}|> \min_{j\neq j'}(|\lambda_{j}(0)-\lambda_{j'}(0)|-|B_{n}^{(j)}|-|B_{n}^{(j')}|)>e^{-\varepsilon n}.
\end{equation}   

The Markov inequality \eqref{non:markov}, applied to the trace, implies 
$$
\max_{E\in [E_{j},E_{j+1}]}|\Delta_{n}(E)| \geq \frac{E_{j+1}-E_{j}}{2n^{2}} \max_{E\in [E_{j},E_{j+1}]}|\Delta_{n}'(E)|
$$
then \eqref{non:05} gives 
$$
|\Delta_{n}(E'_{j})| > \frac{e^{-\varepsilon n}}{2n^{2}}\max_{E\in[E_{j},E_{j+1}]}|\Delta'_{n}(E)| \geq \frac{e^{-\varepsilon n}}{2n^{2}}\max(|\Delta'_{n}(E_{j})|,|\Delta'_{n}(E_{j+1})|).
$$
Last's estimate \eqref{non:03} and the upper bound on the bandwidth \eqref{non:11} give 
\begin{equation}\label{non:01}
|\Delta_{n}(E'_{j})|> \frac{e^{-\varepsilon n}}{2n^{2}}(1+\sqrt{5})e^{(\max(\gamma(\lambda_{j}(0)),\gamma(\lambda_{j+1}(0)))-\varepsilon)n}.
\end{equation} 
In the cases where $2\leq j\leq n-1$, \eqref{non:01} implies \eqref{non:12}, so for all $0\leq g\leq \gamma(\lambda_{j}(0))-\varepsilon$, we have 
$$
|\Delta_{n}(E'_{j-1})|, |\Delta_{n}(E'_{j})| > 2\cosh(ng)
$$
which implies $\lambda_{j}(g)\in (E'_{j-1},E'_{j})$ for all such $g$. 

Consider $\lambda_{1}(g)$. In the case that the eigenvalue moves to the right, i.e., $\lambda_{1}(0)>E_{1}$ we use the lower bound \eqref{non:01} on $|\Delta'_{n}(E'_{1})|$. In the case that the eigenvalue moves to the left, then the eigenvalue is real $\lambda_{1}(g)\in\R$ for every non-negative $g\geq0$ since the trace has no turning points below $E'_{1}$. It remains to appropriately scale the errors in the exponents to match the statement of the lemma. 
\end{proof}

\itshape{Address:} \scshape{School of Mathematical Sciences, Queen Mary University of London, London E1 4NS, United Kingdom.} 

\itshape{E-mail:} \scshape{l.haeming@qmul.ac.uk.}

\end{document}